\documentclass{amsart}
\usepackage{verbatim}
\usepackage{paralist}
\usepackage{leftidx}
\usepackage{dsfont}
\usepackage{amsthm}
\usepackage{amsmath}
\usepackage{amssymb}
\usepackage{amscd}
\usepackage{graphicx}
\usepackage{setspace}
\usepackage[all]{xy}
\usepackage{mathrsfs} 
\usepackage{stmaryrd}
\usepackage{hyperref}
\title[Uniform Manin--Mumford over function fields]{A uniform quantitative Manin--Mumford theorem for curves over function fields}
\author{Nicole Looper}
\address{Department of Mathematics, Statistics, and Computer Science, University of Illinois at Chicago, 851 Morgan Street, IL 60607, Chicago, USA}
\email{nrlooper@uic.edu}

\author{Joseph Silverman}
\address{Mathematics Department, Brown University, 151 Thayer Street, RI 02912, Providence, USA}
\email{joseph\_silverman@brown.edu}

\author{Robert Wilms}
\address{Laboratoire de Math\'ematiques Nicolas Oresme, Universit\'e de Caen--Normandie BP 5186, 14032, Caen Cedex, France}
\email{robert.wilms@unicaen.fr}

\thanks{The first author was supported by NSF grant DMS-1803021. The second author was partially supported by Simons Collaboration Grant $\#712332$. The third author was supported by the Swiss National Science Foundation grant ``Diophantine Equations: Special Points, Integrality, and Beyond'' (n$^\circ$ 200020\_184623).}

\date{\today}
\numberwithin{equation}{section}
\newtheorem{Thm}{Theorem}[section]
\newtheorem{Lem}[Thm]{Lemma}
\newtheorem{Cor}[Thm]{Corollary}
\newtheorem{Pro}[Thm]{Proposition}
\begin{document}

\begin{abstract}
	We prove that any smooth projective geometrically connected non-isotrivial curve of genus $g\ge 2$ over a one-dimensional function field of any characteristic has at most $16g^2+32g+124$ torsion points for any Abel--Jacobi embedding of the curve into its Jacobian. The proof uses Zhang's admissible pairing on curves, the arithmetic Hodge index theorem over function fields, and the metrized graph analogue of Elkies' lower bound for the Green function. More generally, we prove an explicit Bogomolov-type result bounding the number of geometric points of small N\'eron--Tate height on the curve embedded into its Jacobian.
\end{abstract}
\maketitle
\section{Introduction}

Around 1963 Manin and Mumford independently conjectured that for any algebraically closed field $K$ of characteristic $0$, any smooth projective curve $X$ of genus $g\ge 2$ defined over $K$ and any divisor $D\in\operatorname{Div}^1(X)$ of degree $1$ on $X$, one has
\[\# j_D(X(K))\cap J(K)_{\operatorname{tors}}<\infty,\]
where $j_D\colon X\to J,~P\mapsto P-D$ denotes the Abel--Jacobi embedding of $X$ into its Jacobian $J=\operatorname{Pic}^0(X)$ associated to $D$. This was proved by Raynaud \cite[Th\'eor\`eme I]{Ray83} in 1983. There are several further questions and aspects one then naturally considers.

\subsection*{Positive characteristic}
One might ask whether an analogue of this conjecture is true in positive characteristic. Since for $K=\overline{\mathbb{F}}_p$ every $K$-point of a curve over $K$ is a torsion point, the naive analogue is certainly not true. Scanlon \cite[Proposition 4.4]{Sca01} and later Pink and R\"{o}ssler \cite[Theorem 3.6]{PR04} found and proved a suitable analogue of the Manin--Mumford conjecture in positive characteristic, using separate methods.

\subsection*{Uniformity}
It is reasonable to ask whether the number of torsion points contained in $j_D(X(K))$ is uniformly bounded only in terms of the genus $g$. This has recently been proven (in addition to a uniform Bogomolov-type result) by K\"uhne \cite{Kuh21} in characteristic $0$. In \cite{Yua21}, which appeared after our paper was posted on arXiv, Yuan proves uniform Manin--Mumford and Bogomolov-type results in arbitrary characteristic. However, in neither \cite{Kuh21} nor \cite{Yua21} are the results quantitative.

\subsection*{Quantitative Bounds}
It is interesting to consider the question of explicit upper bounds on the number of torsion points in $j_D(X(K))$. If the curve $X$ is defined over a number field $F$ and $J$ has complex multiplication, Coleman \cite[Theorem A]{Col85} obtained the explicit bound $pg$ for a prime number $p\ge5$ depending on the ramification of $F/\mathbb{Q}$ and the reduction type of $X$. Buium \cite[Theorem A]{Bui96} obtained the bound $p^{4g}3^g(p(2g-2)+6g)g!$ without the condition of complex multiplication.

If $\operatorname{char}(K)=0$, Baker and Poonen \cite{BP01} proved that there are only finitely many points $P\in X(K)$ such that $\#j_P(X(K))\cap J(K)_{\operatorname{tors}}>2$. For general $D\in\operatorname{Div}^1(X)$ we have $\{Q\in X(K):j_D(Q)\in J(K)_{\operatorname{tors}}\}=\{Q\in X(K):j_P(Q)\in J(K)_{\operatorname{tors}}\}$ whenever $D-P$ is torsion in $J(K)$. Thus, in bounding $\#j_D(X(K))\cap J(K)_{\operatorname{tors}}$ for general $D\in\operatorname{Div}^1(X)$, it remains only to bound the size of $j_P(X(K))\cap J(K)_{\operatorname{tors}}$ in the finitely many cases where $\#j_P(X(K))\cap J(K)_{\operatorname{tors}}>2$.\\

In this paper we address all three of the above questions for non-isotrivial curves over the function field $K=k(B)$ of any smooth projective connected curve $B$ over any algebraically closed field $k$.
A curve $X$ over $K$ is called isotrivial if there is a finite field extension $K'$ of $K$ and a curve $C$ defined over $k$ such that $X\otimes_{K} K'\cong C\otimes_{k}K'$.
\begin{Thm}\label{mainthm}
	Let $k$ be any algebraically closed field, $B$ any smooth projective connected curve over $k$ and $K=k(B)$ its function field. For any smooth projective geometrically connected non-isotrivial curve $X$ of genus $g\ge 2$ defined over $K$ and any Abel--Jacobi embedding
	\[j_D\colon X\to J,\quad x\mapsto x-D\]
	of $X$ into its Jacobian $J=\operatorname{Pic}^0(X)$, where $D\in\operatorname{Div}^1(X)$ is any divisor of degree $1$, the number of torsion points in $j_D(X(\overline{K}))$ is uniformly bounded by
	\[\# j_D(X(\overline{K}))\cap J(\overline{K})_{\operatorname{tors}}\le c(g)\le 16g^2+32g+124,\]
	where $c(2)=76$, $c(3)=231$ and $c(g)=\left\lfloor\frac{16g^4+37g^2-28g-1}{(g-1)^2}\right\rfloor$ for $g\ge 4$.
	
	Moreover, if $X$ has everywhere potentially good reduction, we may replace $c(g)$ by $1$. If only $J$ has everywhere potentially good reduction, we may replace $c(g)$ by $c^{tr}(g)$ with $c^{tr}(2)=11$ and $c^{tr}(g)=\left\lfloor \frac{4g^3-4g^2+2g+1}{g-1}\right\rfloor\le 4g^2+3$ for $g\ge 3$. If $\operatorname{char}~k=0$ or $X$ is hyperelliptic, we may replace $c(g)$ and $c^{tr}(g)$ by $\left\lfloor\frac{c(g)+1}{2}\right\rfloor$ and $\left\lfloor\frac{c^{tr}(g)+1}{2}\right\rfloor$ for $g\ge 3$.
\end{Thm}
As a consequence we can bound the cardinality of $j_D(X(\overline{K}))\cap J(\overline{K})_{\operatorname{tors}}$ for every field $K$ and any curve $X/K$ that cannot be defined over an algebraic extension of the prime field of $K$.
\begin{Cor}\label{cor}
	Let $K$ be any algebraically closed field and $X$ any smooth projective connected curve of genus $g\ge 2$ over $K$. Assume that $X$ cannot be defined over any algebraic extension of the prime field of $K$. For every divisor $D\in\operatorname{Div}^1(X)$ of degree $1$,
	\[j_D(X(K))\cap J(K)_{\operatorname{tors}}\le c(g),\]
	where $J=\operatorname{Pic}^0(X)$ and $c(g)$ as in Theorem \ref{mainthm}.
\end{Cor}
If $X/K$ is definable over an algebraic extension of a finite field $\mathbb{F}_p$, the intersection $j_D(X(K))\cap J(K)_{\operatorname{tors}}$ is infinite as every point in $J(\overline{\mathbb{F}}_p)$ is a torsion point. Thus, only the case of curves definable over a number field remains open.

We will deduce Theorem \ref{mainthm} from the following more general result on the geometric Bogomolov conjecture, which bounds the number of geometric points $P\in X(\overline{K})$ of small N\'eron--Tate height $h_{\operatorname{NT}}(j_D(P))$ on the Jacobian $J$. To recall the definition of $h_{\operatorname{NT}}$, let $\Theta$ be a theta divisor on $J$ and $\mathcal{L}=\mathcal{O}_J(\Theta)\otimes [-1]^*\mathcal{O}_J(\Theta)$ the associated ample and symmetric line bundle. If $h_{\mathcal{L}}$ is a Weil height associated to $\mathcal{L}$, we set $h_{\operatorname{NT}}(x)=\lim_{n\to \infty}\frac{h_{\mathcal{L}}(nx)}{n^2}$. 

\begin{Thm}\label{thm_bogomolov}
	Let $k$ be any algebraically closed field, $B$ any smooth projective connected curve over $k$ and $K=k(B)$ its function field. Further, let $g\ge 2$ be an integer and $\epsilon\in \left[0,\max\left(\frac{1}{8g},\frac{1}{9(g-1)}\right)\right)$ a real number. For any smooth projective geometrically connected non-isotrivial curve $X$ of genus $g$ defined over $K$ and any divisor $D\in\operatorname{Div}^1(X)$ of degree $1$, 
	\[\# \left\{P\in X(\overline{K})~|~h_{\operatorname{NT}}(j_D(P))\le \epsilon \omega_a^2\right\}\le c(g,\epsilon),\]
	where $h_{\operatorname{NT}}$ denotes the N\'eron~--Tate height on the Jacobian $J=\operatorname{Pic}^0(X)$, $\omega_a^2$ is the self-pairing of the admissible relative dualizing sheaf $\omega_a$ on $X$, and the constant $c(g,\epsilon)$ is given by $c(2,\epsilon)=\left\lfloor \frac{75}{1-9\epsilon}\right\rfloor+1$, $c(3,\epsilon)=\left\lfloor\frac{3920}{17(1-18\epsilon)}\right\rfloor +1$ and $c(g,\epsilon)=\left\lfloor\frac{16g^4+36g^2-26g-2}{(g-1)^2(1-\min\{9(g-1),8g\}\epsilon)}\right\rfloor+1$ for $g\ge 4$.
	
	Moreover, if $J$ has everywhere potentially good reduction, we may replace $c(g,\epsilon)$ by $c^{tr}(g,\epsilon)$ with $c^{tr}(2,\epsilon)=\left\lfloor \frac{10}{1-9\epsilon}\right\rfloor+1$ and $c^{tr}(g,\epsilon)=\left\lfloor \frac{4g^3-4g^2+g+2}{(g-1)(1-\min\{9(g-1),8g\}\epsilon)}\right\rfloor+1$ for $g\ge 3$. If $\operatorname{char}~k=0$ or $X$ is hyperelliptic, we may replace $c(g,\epsilon)$ and $c^{tr}(g,\epsilon)$ by $\left\lfloor\frac{c(g,\epsilon)+1}{2}\right\rfloor$ and $\left\lfloor\frac{c^{tr}(g,\epsilon)+1}{2}\right\rfloor$ for $g\ge 3$.
	Finally, if $\epsilon\in\left[0,\frac{1}{4(g^2-1)}\right)$ and $X$ has everywhere potentially good reduction, we may replace $c(g,\epsilon)$ by $1$.
\end{Thm}

The proof makes heavy use of Zhang's admissible pairing introduced in \cite{Zha93}. Replacing $K$ by a finite field extension $K'$ and $B$ by its normalization in $K'$, we may and will assume that $X$ has semistable reduction over $B$. If $F$ denotes the divisor of points $P_1,\dots,P_s\in X(\overline{K})$, which are mapped to points of small N\'eron--Tate height by $j_D$, then the non-negativity of the N\'eron--Tate height of $s\omega-(2g-2)F$ gives a bound of the admissible self-intersection $\omega_a^2$ in terms of the Green functions on the metrized reduction graphs on the pairs of the points of $F$. By an application of the arithmetic Hodge index theorem on $X^2$, we can bound $\omega_a^2$ from below by the Zhang invariant $\varphi$ of the metrized reduction graphs. We also obtain an upper bound for the sum of the Green functions in the pairs of points of $F$ in terms of $\varphi$ through the metrized graph analogue of Elkies' bound for the Green function and an estimate of the supremum of the Green function on metrized graphs. We may summarize these bounds by
\begin{align*}
	\tfrac{\max(2,g-1)}{2g+1}\varphi(X)\le \omega_a^2&\le -\tfrac{4(g-1)g}{s(s-1)(1-4(g^2-1)\epsilon)}\sum_{j\neq k}^s \sum_{v\in|B|}g_v(R_v(P_j),R_v(P_k))\\
	&\le \tfrac{4(g-1)gc'(g)}{(s-1)(1-4(g^2-1)\epsilon)}\varphi(X),
\end{align*}
where $\varphi(X)=\sum_{v\in |B|}\varphi(\Gamma_v(X))$ is the sum of Zhang's invariant $\varphi$ for all metrized reduction graphs $\Gamma_v(X)$ of $X$ at the closed points $v\in |B|$, $g_v$ denotes the Green function on $\Gamma_v(X)$ and $R_v\colon X(\overline{K})\to \Gamma_v(X)$ is the reduction map. The constant $c'(g)$ is explicitly given in Lemma \ref{lem_sup}. From this bound it is possible to deduce a bound of $s$ in terms of $g$ and $\epsilon$ independently of $K$.

The lower bound $\omega_a^2\ge \frac{g-1}{2g+1}\varphi(X)$ of the admissible self-intersection number $\omega_a^2$ was already obtained in the case of number fields by the third author in \cite[Theorem 1.2]{Wil19}. It can be proven over function fields by exactly the same arguments. Since the arguments are widely dispersed over the literature and are often only formulated for number fields, however, we decided to include a more self-contained proof here.

\subsection*{Outline}
In Section \ref{sec_graphs} we recall the theory of polarized metrized graphs. We give the definitions of the required invariants and we bound the Green function of a polarized metrized graph in terms of Zhang's invariant $\varphi$. We recall the notion of adelic metrics in Section \ref{sec_adelic}. In Sections \ref{sec_admissible} and \ref{sec_jacobian} we discuss the admissible adelic metrics for line bundles on the curve $X$ and on its Jacobian $J$ as well as one of their relations. In Section \ref{sec_hodgeindex} we deduce the lower bound for the admissible self-intersection number $\omega_a^2$ from the arithmetic Hodge index theorem on $X^2$. Finally, we give the proof of Theorems \ref{mainthm} and \ref{thm_bogomolov} in Section \ref{sec_proof} and the proof of Corollary \ref{cor} in Section \ref{sec_cor}.

\section{Polarized metrized graphs}\label{sec_graphs}
In this section we discuss the notion of polarized metrized graphs as introduced by Zhang \cite{Zha93} and Chinburg--Rumely \cite{CR93}, see also \cite{BR07}. In particular, we are interested in the sum of the canonical Green function of a polarized metrized graph over pairs of points in a fixed set of points and we will estimate it in terms of Zhang's $\varphi$-invariant.

A \emph{metrized graph} is a compact and connected nonempty metric space $\Gamma$, which is locally around any point $p\in\Gamma$ isometric to the star-shaped set
\[S(n_p,r_p)=\left\{z\in \mathbb{C}~|~z=t e^{2\pi i k/n_p}\text{ for some } 0\le t\le r_p\text{ and } k\in\mathbb{Z}\right\},\]
for some integer $n_p\ge 1$, the \emph{valency} of $p$, and some real number $r_p>0$, or to a single point, in which case we set $n_p=0$. A \emph{vertex set} for $\Gamma$ is a nonempty finite set $V\subseteq \Gamma$, such that $\{p\in\Gamma~|~n_p\neq 2\}\subseteq V$ and the closure of any connected component of $\Gamma\setminus V$ intersects $V$ in exactly two points. We fix a vertex set $V$ and we write $e_1,\dots,e_r$ for the connected components of $\Gamma\setminus V$, which are just open line segments. Let $\ell(e)$ be the length of any line segment $e$ and write 
\[\delta(\Gamma)=\sum_{i=1}^r \ell(e_i)\]
for the total length of $\Gamma$.

Let us describe how we associate a metrized graph to a connected classical graph $(V,E)$, consisting of a finite set $V$ of vertices and a non-empty set of edges $E\subseteq V\times V$, which is additionally equipped with a length function $\ell\colon E\to \mathbb{R}_{> 0}$. For $i\in \{1,2\}$ we write $\operatorname{pr}_i\colon E\to V$ for the projection to the $i$-th factor. For each $e\in E$ let $I_e$ be a metric space equipped with an isometry $\varphi_e\colon [0,\ell(e)]\to I_e$. We define an equivalence relation $\sim$ on the disjoint union $\bigsqcup_{e\in E}I_e$ by setting $x\sim y$ for two points $x\in I_e$ and $y\in I_{e'}$ if there exist $i,j\in\{1,2\}$ and a vertex $v\in \operatorname{pr}_i(e)\cap\operatorname{pr}_j(e')$ such that
\[x=\varphi_e((i-1)\cdot\ell(e))\qquad \text{and}\qquad y=\varphi_{e'}((j-1)\cdot\ell(e')).\]
In other words, the end points of the edges represented by the metric spaces $I_e$ are identified if they meet in the same vertex.
Now the metrized graph associated to $(V,E)$ is given by the metric space
\[\Gamma_{(V,E)}=\left(\bigsqcup_{e\in E}I_e\right)/\sim.\]
Although we started with a directed graph $(V,E)$, one can check that $\Gamma_{(V,E)}$ does not depend on the choice of the direction. We have a canonical inclusion $V\subseteq \Gamma_{(V,E)}$ by sending $v\in V$ to the class of any $x\in I_e$ such that there is an $i\in\{1,2\}$ with $v=\operatorname{pr}_i(e)$ and $x=\varphi_e((i-1)\cdot\ell(e))$.
If $(V,E)$ is a connected graph with $E=\emptyset$, we set $\Gamma_{(V,E)}$ to be a point. We again get a canonical inclusion $V\subseteq \Gamma_{(V,E)}$.

Next, we would like to define the tangent space of $\Gamma$ at a point $p\in\Gamma$. By a path in $\Gamma$ we mean a length-preserving continuous map $\gamma\colon [0,l]\to \Gamma$ for some $l>0$. We say that two paths $\gamma_1,\gamma_2$ are equivalent, written as $\gamma_1\sim \gamma_2$, if there exists an $\epsilon>0$ with $\gamma_1(t)=\gamma_2(t)$ for all $0\le t<\epsilon$. The tangent space $T_p(\Gamma)$ of $\Gamma$ at $p$ is given by
\[T_p(\Gamma)=\{\gamma \text{ path in } \Gamma \text{ with } \gamma(0)=p\}/\sim.\]
For $v\in T_p(\Gamma)$, represented by a path $\gamma$, and $f:\Gamma\to\mathbb{R}$, the \emph{one-sided directional derivative} at $p$ is defined by
\[d_v f(p)=\lim_{t\to 0^+}\frac{f(\gamma(t))-f(p)}{t}\]
if the limit exists.
For the functions
\begin{align*}
	\operatorname{Zh}(\Gamma)=\{f\colon \Gamma\to\mathbb{R}~|~&f \text{ continuous, piecewise } C^2 \text{ and} \\
	&d_v f(p) \text{ exists for all } p\in\Gamma, v\in T_p(\Gamma)\}
\end{align*}
we recall Zhang's definition \cite[Appendix]{Zha93} of the Laplace operator
\[\Delta(f)=-f''dx-\sum_{p\in\Gamma}\left(\sum_{v\in T_p(\Gamma)}d_v f(p)\right)\delta_p.\]
Here, $\delta_p$ is the Dirac measure at $p$ and $\Delta(f)$ is considered as a signed measure on $\Gamma$.

For any bounded signed measure $\mu$ of total mass $1$ on $\Gamma$ there exists a unique continuous and symmetric function $g_\mu\colon\Gamma\times\Gamma\to \mathbb{R}$, called the \emph{Green function} associated to $\mu$, such that $g_\mu(x,\cdot)\in \operatorname{Zh}(\Gamma)$ is characterized by
\[\Delta_y g_\mu(x,y)=\delta_x(y)-\mu(y) \quad\text{and}\quad \int_{\Gamma}g_\mu(x,y)\mu(y)=0\]
for all $x\in \Gamma$. Further, $r(p,q)=g_{\delta_p}(q,q)$ is called the \emph{resistance function} and measures the effective resistance between $p$ and $q$ if we consider $\Gamma$ as a network, where $\ell(e)$ is the resistance along any edge $e$.

By a \emph{divisor} on $\Gamma$ we mean any finite formal sum $D=\sum_{p\in\Gamma} m_p p$ for some integers $m_p\in\mathbb{Z}$. A \emph{polarization} on $\Gamma$ is an effective divisor $K=\sum_{p\in\Gamma}(n_p+2m_p-2)p$ with $m_p\ge 0$ and $m_p=0$ for all but finitely many $p\in\Gamma$. We call $(\Gamma,K)$ a \emph{polarized metrized graph} of genus $g=\frac{1}{2}\deg K+1$. In the following we fix a polarization $K$ on $\Gamma$. For any function $f\colon \Gamma\to \mathbb{R}$ and any divisor $D=\sum_{p\in\Gamma} m_p p$ on $\Gamma$ we write $f(D)=\sum_{p\in\Gamma}m_p f(p)$.

By \cite[Theorem 3.2]{Zha93} there exists a unique positive measure $\mu_K$ of volume $1$ on $\Gamma$ such that
\begin{align}\label{equ_c}
	c(\Gamma,K)=g_{\mu_K}(x,K)+g_{\mu_K}(x,x)
\end{align}
is independent of $x\in\Gamma$. We call $g=g_{\mu_K}$ the \emph{canonical Green function} associated to $(\Gamma,K)$. 
Let us recall the following invariants of $(\Gamma, K)$: The $\epsilon$-invariant
\[\epsilon(\Gamma,K)=\int_{\Gamma} g(x,x)((2g-2)\mu_K+\delta_{K})\]
defined by Zhang in \cite[Theorem 4.4]{Zha93} and the $\varphi$-invariant
\[\varphi(\Gamma, K)=-\tfrac{1}{4}\delta(\Gamma)+\tfrac{1}{4}\int_\Gamma g(x,x)((10g+2)\mu_K-\delta_K)\]
defined by Zhang in \cite[Theorem 1.3.1]{Zha10}. 
The main goal of this section is to prove the following proposition.
\begin{Pro}\label{pro_elkies}
	Let $(\Gamma,K)$ be any polarized metrized graph and $x_1,\dots, x_s\in \Gamma$ any points. Then
	\[\sum_{j\neq k}^s g(x_j,x_k)\ge -sc'(g) \varphi(\Gamma,K),\]
	where $c'(2)=15/4$, $c'(3)=140/51$ and $c'(g)=\frac{8g^4+18g^2-13g-1}{2g(2g+1)(g-1)^2}$ for $g\ge 4$.
	If $\Gamma$ is a tree, we may replace $c'(g)$ by $c'^{tr}(g)=\frac{2g^2-3g+2}{2g(2g-2)}$.
\end{Pro}
Before we prove the proposition we state two lemmas.
The first lemma expresses the invariant $c(\Gamma,K)$ in terms of the invariants $\varphi(\Gamma,K)$, $\delta(\Gamma)$ and $\epsilon(\Gamma, K)$.
\begin{Lem}\label{lem_c}
	Let $(\Gamma,K)$ be any polarized metrized graph. The invariant $c(\Gamma,K)$ can be expressed by
	\[c(\Gamma,K)=\tfrac{1}{12g}(4\varphi(\Gamma,K)+\delta(\Gamma)+\epsilon(\Gamma,K)).\]
\end{Lem}
\begin{proof}
	If we integrate Equation (\ref{equ_c}) with respect to $\mu_K$, we obtain by the definition of the Green function $g$ that
	\[c(\Gamma,K)=\int_\Gamma g(x,x)\mu_K.\]
	Hence, by the defining equations for the invariants $\epsilon(\Gamma,K)$ and $\varphi(\Gamma,K)$ we get
	\begin{align*}
		\epsilon(\Gamma,K)&=(2g-2)c(\Gamma,K)+\int_{\Gamma} g(x,x)\delta_{K},\\
		4\varphi(\Gamma, K)+\delta(\Gamma)&=(10g+2)c(\Gamma,K)-\int_\Gamma g(x,x)\delta_K.
	\end{align*}
	Taking the sum of both equations yields
	\[12g c(\Gamma,K)=4\varphi(\Gamma,K)+\delta(\Gamma)+\epsilon(\Gamma,K),\]
	which proves the lemma after dividing by $12g$ on both sides.
\end{proof}
Analogously to Equation (\ref{equ_c}), Baker and Rumely proved in \cite[Theorem 14.1]{BR07} that there also exists a unique bounded signed measure $\mu_0$ of total mass $1$ on $\Gamma$ such that 
\[\tau(\Gamma):=g_{\mu_0}(x,x)\]
is a constant depending only on the metrized graph $\Gamma$. They show that it can be expressed by
\begin{align}\label{equ_tau}
	\tau(\Gamma)=g_{\mu_0}(x,y)+\tfrac{1}{2}r(x,y)
\end{align}
for any points $x,y\in\Gamma$.
De Jong \cite[Proposition 9.2]{dJo18} further proved the relation
\begin{align}\label{equ_tau-dejong}
	\tau(\Gamma)=\tfrac{1}{12}(\delta(\Gamma)+4\varphi(\Gamma,K)-2\epsilon(\Gamma,K)).
\end{align}
Moreover, Baker and Rumely have proven the following analogue of Elkies' lower bound for the Green function in \cite[Proposition 13.7]{BR07}
\begin{align}\label{equ_br-elkies}
	\sum_{j\neq k}^s g_\mu(x_j,x_k)\ge -s\cdot \sup_{x\in\Gamma}g_\mu(x,x),
\end{align}
where $\mu$ is any measure of volume $1$ on $\Gamma$ and $x_1,\dots,x_s\in \Gamma$ are any points. Although their statement differs from the above inequality, their proof gives exactly (\ref{equ_br-elkies}). For $s=2$ and $\mu=\mu_0$ this implies $g_{\mu_0}(x,y)\ge -\tau(\Gamma)$. In combination with Equations (\ref{equ_tau}) and (\ref{equ_tau-dejong}) we get
\begin{align}\label{equ_rbound}
	r(x,y)\le 4\tau(\Gamma)=\tfrac{1}{3}(\delta(\Gamma)+4\varphi(\Gamma,K)-2\epsilon(\Gamma,K))
\end{align}
for any points $x,y\in\Gamma$.

Next, we note that we have 
\begin{align}\label{inequ_cinkir}
	\delta(\Gamma)\le c_C(g)\varphi (\Gamma, K)
\end{align}
with $c_C(2)=27$, $c_C(3)=288/17$ and $c_C(g)=\tfrac{2g(7g+5)}{(g-1)^2}$ for $g\ge 4$ due to \c{C}inkir \cite[Theorem 2.11]{Cin11}, respectively \cite{Cin15} for $g=3$. If $\Gamma$ is a tree, we may even replace $c_C(g)$ by $c_C^{tr}(g)=g/(2g-2)$.
The following Lemma gives a bound for the canonical Green function in terms of the $\varphi$-invariant.
\begin{Lem}\label{lem_sup}
	For any polarized metrized graph $(\Gamma,K)$ and any $x,y\in\Gamma$, one has
	\[g(x,y)\le \frac{(4g-3)\delta(\Gamma)+(16g-12)\varphi(\Gamma,K)-(8g-3)\epsilon(\Gamma,K)}{12g}\le c'(g)\varphi(\Gamma,K),\]
	with $c'(g)$ as in Proposition \ref{pro_elkies}.
	If $\Gamma$ is a tree, we may replace $c'(g)$ by $c'^{tr}(g)$ with $c'^{tr}(g)$ as in Proposition \ref{pro_elkies}.
\end{Lem}
\begin{proof}
	Let us first show that $\sup_{y\in\Gamma} g_{\mu}(x,y)=g_{\mu}(x,x)$ for any positive measure $\mu$ on $\Gamma$ of volume $1$. If we fix two points $x,z\in \Gamma$, then
	\[\Delta_y(g_\mu(x,y)-g_\mu(z,y))=\delta_x(y)-\delta_z(y).\]
	Thus, $f(y)=g_\mu(x,y)-g_\mu(z,y)$ is a piecewise linear function. For every $p\in\Gamma\setminus\{x,z\}$, one has the identity $\sum_{v\in T_p(\Gamma)}d_vf(p)=0$. Hence at any $p\in\Gamma\setminus\{x,z\}$, the function $f$ is either locally constant or has no local extremum. Further, $\sum_{v\in T_x(\Gamma)}d_vf(x)=-1$ and $\sum_{v\in T_z(\Gamma)}d_vf(z)=1$. We conclude that $f$ attains its global maximum at $x$. Thus,
	\[g_\mu(x,x)-g_\mu(z,x)=f(x)=\sup_{y\in \Gamma}f(y)\ge \int_{\Gamma}f(y)\mu(y)=0,\] where the inequality uses the positivity of $\mu$.
	This proves that $g_{\mu}(z,x)\le g_{\mu}(x,x)$ for all $x,z\in \Gamma$.
	Hence, by the positivity of $\mu_K$ (see (\ref{equ_c})), it is enough to bound $\sup_{x\in\Gamma} g(x,x)$. 
	
	By \cite[Lemma 4.1]{Mor96}, we have
	\[g(x,x)=c(\Gamma,K)-g(x,K)=c(\Gamma,K)+\frac{r(x,K)-\epsilon(\Gamma,K)}{2g}.\]
	As $K$ is effective of degree $2g-2$, an application of the bound (\ref{equ_rbound}) yields
	\[\sup_{x,y\in\Gamma}g(x,y)\le c(\Gamma,K)+\frac{(2g-2)\delta(\Gamma)+(8g-8)\varphi(\Gamma,K)-(4g-1)\epsilon(\Gamma,K)}{6g}.\]
	Now the first inequality in the lemma follows by Lemma \ref{lem_c}.
	
	To bound this expression in terms of $\varphi(\Gamma,K)$, the direct way would be to use \c{C}inkir's inequality (\ref{inequ_cinkir}) and $\epsilon(\Gamma,K)\ge 0$. To obtain a slightly better bound, we also recall from \c{C}inkir's work \cite[Theorem 2.13]{Cin11} that the invariant 
	\[\lambda(\Gamma,K)=\tfrac{g-1}{6(2g+1)}\varphi(\Gamma,K)+\tfrac{1}{12}(\epsilon(\Gamma,K)+\delta(\Gamma))\]
	satisfies $(8g+4)\lambda(\Gamma,K)\ge g\delta(\Gamma)$. This can be rewritten as
	\[-(2g+1)\epsilon(\Gamma,K)\le 2(g-1)\varphi(\Gamma,K)-(g-1)\delta(\Gamma).\]
	Using this, we compute the estimate
	\begin{align*}
		&\frac{(4g-3)\delta(\Gamma)+(16g-12)\varphi(\Gamma,K)-(8g-3)\epsilon(\Gamma,K)}{12g}\\
		&\le \frac{(3g-2)\delta(\Gamma)+(16g^2-10g-2)\varphi(\Gamma,K)}{4g(2g+1)}\\
		&\le \frac{(3g-2)c_C(g)+16g^2-10g-2}{4g(2g+1)}\varphi(\Gamma,K).
	\end{align*}
	If we denote the coefficient on $\varphi(\Gamma,K)$ in the last expression by $c'(g)$, we can explicitly write it as
	$c'(2)=15/4$, $c'(3)=140/51$ and $c'(g)=\frac{8g^4+18g^2-13g-1}{2g(2g+1)(g-1)^2}$ for $g\ge 4$.
	If $\Gamma$ is a tree, then we have $2\varphi(\Gamma,K)=\delta(\Gamma)+\epsilon(\Gamma,K)$ by \cite[Equation~(1.4)]{dJS18}, as the Jacobian $\operatorname{Jac}(\Gamma)$ of a tree $\Gamma$ is trivial. We compute
	\begin{align*}
		&\frac{(4g-3)\delta(\Gamma)+(16g-12)\varphi(\Gamma,K)-(8g-3)\epsilon(\Gamma,K)}{12g}=\frac{(2g-1)\delta(\Gamma)-\varphi(\Gamma,K)}{2g}\\
		& \le \frac{(2g-1)c^{tr}_C(g)-1}{2g}\varphi(\Gamma,K)=\frac{2g^2-3g+2}{2g(2g-2)}\varphi(\Gamma,K).\qedhere
	\end{align*}
	
\end{proof}
Proposition \ref{pro_elkies} follows by combining the bound in (\ref{equ_br-elkies}) for $\mu=\mu_K$ with Lemma \ref{lem_sup}.

\section{Adelic metrics}\label{sec_adelic}
We recall the notion of adelic metrics on line bundles in this section. We refer to \cite[Section 1]{Zha95} and \cite[Section 2]{Car20} for additional details. Let $K=k(B)$ be the function field of a smooth projective connected curve $B$ defined over an algebraically closed field $k$. Let $Y$ be a smooth projective variety over $K$ and $\mathcal{L}$ any line bundle on $Y$. For any closed point $v\in |B|$ we denote $\overline{K}_v$ for the algebraic closure of the completion of $K$ with respect to $v$. This is a valuation field and we denote its valuation ring by $\mathcal{O}_{\overline{K}_v}$. We write $Y_v$ and $\mathcal{L}_v$ for the pullbacks of $Y$ and $\mathcal{L}$ induced by the embedding $K\to\overline{K}_v$.

By a \emph{metric} $\|\cdot\|$ on $\mathcal{L}_v$ we mean a collection of $\overline{K}_v$-norms $\|\cdot\|_y$ on $y^*\mathcal{L}$ for every $y\in Y(\overline{K}_v)$. 
An important example is the model metric $\|\cdot\|_{\widetilde{\mathcal{L}_v}}$ associated to any projective flat model $(\widetilde{Y},\widetilde{\mathcal{L}_v})$ of $(Y,\mathcal{L}_v^{\otimes e})$ over $\mathcal{O}_{\overline{K}_v}$ for any integer $e>0$, which is given by
\[\|l\|_{\widetilde{\mathcal{L}_v},y}=\inf_{a\in \overline{K}_v}\left\{|a|^{1/e}~\left|~l\in a\widetilde{y}^*\widetilde{\mathcal{L}_v}\right.\right\}\]
for any $y\in Y(\overline{K}_v)$, where $\widetilde{y}\in \widetilde{Y}(\mathcal{O}_{\overline{K}_v})$ denotes its unique extension. We call a metric $\|\cdot\|$ on $\mathcal{L}_v$ \emph{continuous and bounded} if there exists a model $(\widetilde{Y},\widetilde{\mathcal{L}_v})$, such that $\log\frac{\|\cdot\|}{\|\cdot\|_{\widetilde{\mathcal{L}_v}}}$ is continuous and bounded.

We call a collection of metrics $\|\cdot\|=\{\|\cdot\|_v~|~v\in |B|\}$ an \emph{adelic metric} on $\mathcal{L}$ if $\|\cdot\|_v$ is a continuous and bounded metric on $\mathcal{L}_v$ for all $v\in |B|$ and there is an open non-empty subset $U\subseteq B$ and a model $(\widetilde{Y},\widetilde{\mathcal{L}})$ of $(Y,\mathcal{L})$ over $U$ such that $\|\cdot\|_v=\|\cdot\|_{\widetilde{\mathcal{L}}_v}$ for every $v\in U$, where $\widetilde{\mathcal{L}}_v$ denotes the pullback of $\widetilde{\mathcal{L}}$ along $\operatorname{Spec}(\mathcal{O}_{\overline{K}_v})\to U$. A pair $\overline{\mathcal{L}}=(\mathcal{L},\|\cdot\|)$ of a line bundle $\mathcal{L}$ and an adelic metric $\|\cdot\|$ on $\mathcal{L}$ is called an \emph{adelic line bundle}. An isometry between two adelic line bundles $\overline{\mathcal{L}}_1=(\mathcal{L}_1,\|\cdot\|_1)$ and $\overline{\mathcal{L}}_2=(\mathcal{L}_2,\|\cdot\|_2)$ is an isomorphism of line bundles $\mathcal{L}_1\cong\mathcal{L}_2$, which induces an isometry between the metrized line bundles $(\mathcal{L}_{1,v},\|\cdot\|_{1,v})$ and $(\mathcal{L}_{2,v},\|\cdot\|_{2,v})$ on $Y_v$ for all $v\in |B|$.

We call an adelic line bundle $\overline{\mathcal{L}}$ \emph{nef} if there is a sequence of models $(\widetilde{Y}_n,\widetilde{\mathcal{L}}_n)$ of $(Y,\mathcal{L}^{\otimes e_n})$ for some $e_n\in\mathbb{Z}_{\ge 1}$ with $\widetilde{\mathcal{L}}_n$ nef on $\widetilde{Y}_n$, such that $\log\frac{\|\cdot\|_{\widetilde{\mathcal{L}}_n,v}}{\|\cdot\|_v}$ converges to $0$ uniformly in $Y(\overline{K}_v)$ for all $v\in |B|$ and there is a non-empty open subset $U\subseteq B$ such that $\log\frac{\|\cdot\|_{\widetilde{\mathcal{L}}_n,v}}{\|\cdot\|_v}$ is identically $0$ for all $v\in|U|$.  Furthermore, we call an adelic line bundle $\overline{\mathcal{L}}$ \emph{integrable} if there exist nef adelic line bundles $\overline{\mathcal{L}}_1$ and $\overline{\mathcal{L}}_2$ and an isometry $\overline{\mathcal{L}}\cong \overline{\mathcal{L}}_1\otimes\overline{\mathcal{L}}_2^{\otimes-1}$.
In what follows we will use the additive notation $\overline{\mathcal{L}}_1+\overline{\mathcal{L}}_2:=\overline{\mathcal{L}}_1\otimes \overline{\mathcal{L}}_2$ and $n\overline{\mathcal{L}}:=\overline{\mathcal{L}}^{\otimes n}$ for $n\in \mathbb{Z}$.

Next we define intersection numbers of integrable line bundles. Let $Z\subseteq Y$ be a subvariety of dimension $d=\dim Z$ and let $\overline{\mathcal{L}}_0,\dots,\overline{\mathcal{L}}_d$ be integrable line bundles on $Y$. We choose rational sections $l_0,\dots,l_d$ of $\overline{\mathcal{L}_0},\dots,\overline{\mathcal{L}_d}$ such that $\bigcap_{i=0}^d\operatorname{supp}(\operatorname{div}(l_i))$ does not intersect $Z$. Fix a $v\in |B|$. If the metric $\|\cdot\|_v$ on $\overline{\mathcal{L}}_i$ is induced by a model $(\widetilde{Y}_v,\widetilde{\mathcal{L}}_{i,v})$ of $(Y,\mathcal{L}_i^{\otimes e})$ for all $i$ and some $e>0$, we define the local intersection number at $v$ by the usual intersection number on $\widetilde{Y}_v$
\[\left(\widehat{\operatorname{div}}(l_0)\cdots\widehat{\operatorname{div}}(l_d)\cdot[Z]\right)_v=\operatorname{div}_{\widetilde{Y}_v}(\widetilde{l}_0|_{Y_v})\cdot\ldots\cdot \operatorname{div}_{\widetilde{Y}_v}(\widetilde{l}_d|_{Y_v})\cdot[\widetilde{Z}_v]/e^{d+1},\]
where $\widetilde{Z}_v$ denotes the Zariski closure of $Z_v$ in $\widetilde{Y}_v$ and we write $\widetilde{l}_i$ for the section of $\widetilde{\mathcal{L}}_i$ extending $l_i^{\otimes e}$. In general, we assume the local intersection number $\left(\widehat{\operatorname{div}}(l_0)\cdots\widehat{\operatorname{div}}(l_d)\cdot [Z]\right)_v$ at $v$ varies continuously with respect to the metrics on the line bundles. As the metric on $\overline{\mathcal{L}}_i$ at $v$ is the limit of metrics induced by models for all $v\in |B|$ and all $i$, we can define the global intersection number by
\[\overline{\mathcal{L}}_0\cdots\overline{\mathcal{L}}_d\cdot Z=\sum_{v\in |B|}\left(\widehat{\operatorname{div}}(l_0)\cdots\widehat{\operatorname{div}}(l_d)\cdot [Z]\right)_v.\]
If $\dim Z=0$ we write $\widehat{\deg}(\overline{\mathcal{L}}_0|_Z)=\overline{\mathcal{L}}_0\cdot Z$ and if $Z=Y$ we write $\overline{\mathcal{L}}_0\cdots\overline{\mathcal{L}}_d=\overline{\mathcal{L}}_0\cdots\overline{\mathcal{L}}_d\cdot Z$ by way of abbreviation. We remark that the global intersection number thus defined is independent of the choice of rational sections $l_i$.

If moreover $\operatorname{div}(l_d)$ is a prime divisor on $Y$ such that $l_i|_{\operatorname{div}(l_d)}\neq 0$ for all $i<d$, the intersection number can be computed recursively by the formula
\begin{align}\label{equ_intersection}
	\overline{\mathcal{L}}_0\cdots\overline{\mathcal{L}}_d=\overline{\mathcal{L}}_0|_{\operatorname{div}(l_d)}\cdots\overline{\mathcal{L}}_{d-1}|_{\operatorname{div}(l_d)}-\sum_{v\in|B|}\int_{Y(\overline{K}_v)} \log\|l_d\|_vc_1(\overline{\mathcal{L}}_0)\cdots c_1(\overline{\mathcal{L}}_{d-1}),
\end{align}
where the integral is defined as follows: First, let each adelic line bundle $\overline{\mathcal{L}}_i$ be induced by a model $(\widetilde{Y},\widetilde{\mathcal{L}}_i)$ of $(Y,\mathcal{L}_i^{\otimes e})$. Then $V=\operatorname{div}(\widetilde{l}_d)-e\cdot\overline{\operatorname{div}(l_d)}$ is a Weil divisor $V=\sum_{v\in |B|}V_v$ supported on the closed fibers of $\widetilde{Y}$ and we set
\[\int_{Y(\overline{K}_v)} \log\|l_d\|_vc_1(\overline{\mathcal{L}}_0)\cdots c_1(\overline{\mathcal{L}}_{d-1})=c_1(\widetilde{\mathcal{L}}_0)\cdots c_1(\widetilde{\mathcal{L}}_{d-1})[V_v]/e^{d+1}.\]
In general the integral is defined by continuity and taking limits.

Let $f\colon Y\to S$ be any proper morphism of smooth projective varieties over $K$ and let $\overline{\mathcal{L}}_0,\dots,\overline{\mathcal{L}}_d$ be integrable line bundles on $S$. Then we have a projection formula
\begin{align}\label{equ_projection1}
	\overline{\mathcal{L}}_0\cdots\overline{\mathcal{L}}_d\cdot f_*(Z)=f^*\overline{\mathcal{L}}_0\cdots f^*\overline{\mathcal{L}}_d\cdot Z.
\end{align}
For model metrics this formula follows from the projection formula in classical intersection theory, see for example \cite[Example 2.4.3]{Ful98}. In general it follows by taking limits. In the same way one proves, that if $K\subseteq K'$ is a field extension, $Y_{K'}=Y\otimes_K K'$ the base change of $Y$ over $K'$ and $\alpha\colon Y_{K'}\to Y$ the induced map, then
\begin{align}\label{equ_projeciton-field-extension}
	\overline{\mathcal{L}}_0\cdots\overline{\mathcal{L}}_d\cdot \alpha_*(Z')=\alpha^*\overline{\mathcal{L}}_0\cdots \alpha^*\overline{\mathcal{L}}_d\cdot Z'
\end{align}
for any integrable line bundles $\overline{\mathcal{L}}_0,\dots,\overline{\mathcal{L}}_d$ on $Y$ and any subvariety $Z'\subseteq Y_{K'}$.
If $s=\dim S$ is the dimension of $S$ and $\overline{\mathcal{M}}_0,\dots,\overline{\mathcal{M}}_s$ are integrable line bundles on $S$, we obtain another projection formula
\begin{align}\label{equ_projection2}
	f^*\overline{\mathcal{M}}_0 \cdots f^*\overline{\mathcal{M}}_s\cdot \overline{\mathcal{L}}_1\cdots \overline{\mathcal{L}}_{d-s}=\overline{\mathcal{M}}_0\cdots\overline{\mathcal{M}}_s\cdot \left(c_1(\mathcal{L}_1)\dots c_1(\mathcal{L}_{d-s})[f]\right),
\end{align}
which one can similarly prove first for model metrics via the classical projection formula, and in the general case by taking limits. Here, $c_1(\mathcal{L}_1)\dots c_1(\mathcal{L}_n)[f]$ denotes the multidegree of the generic fiber of $f$ with respect to the line bundles $\mathcal{L}_1,\dots, \mathcal{L}_n$.

We say that an integrable line bundle $\overline{\mathcal{L}}$ on $Y$ is \emph{numerically trivial} if
\[\overline{\mathcal{L}}\cdot \overline{\mathcal{L}}_1\cdots \overline{\mathcal{L}}_d=0\]
for all integrable line bundles $\overline{\mathcal{L}}_1,\dots,\overline{\mathcal{L}}_d$ on $Y$, where $d=\dim Y$. We call two integrable line bundles $\overline{\mathcal{L}}_1$ and $\overline{\mathcal{L}}_2$ \emph{numerically equivalent} if the difference $\overline{\mathcal{L}}_1-\overline{\mathcal{L}}_2$ is numerically trivial. The intersection number $\overline{\mathcal{L}}_0\cdots\overline{\mathcal{L}}_d$ of integrable line bundles $\overline{\mathcal{L}}_0,\dots,\overline{\mathcal{L}}_d$ only depends on their numerical equivalence classes. To any real number $r\in\mathbb{R}$ and any place $w\in |B|$ we can associate the integrable line bundle $\mathcal{O}_Y(r,w)=(\mathcal{O}_Y,\|\cdot\|_{(r,w)})$ on $Y$, where $\mathcal{O}_Y$ is the trivial bundle on $Y$ and $\|\cdot\|_{(r,w),v}=e^{-\delta_{v,w}r}\|\cdot\|_{v}$ with $\|\cdot\|_{v}$ the canonical metric on $\mathcal{O}_{Y,v}$ and $\delta_{v,w}$ the Kronecker delta. Applying Equation (\ref{equ_intersection}) to the constant section $1\in H^0(Y,\mathcal{O}_Y)$, we get
\[\mathcal{O}_Y(r,w)\cdot \overline{\mathcal{L}}_1\cdots\overline{\mathcal{L}}_d=r\cdot (\mathcal{L}_1\cdots \mathcal{L}_d)\]
for any integrable line bundles $\overline{\mathcal{L}}_1,\dots,\overline{\mathcal{L}}_d$, where $(\mathcal{L}_1\cdots \mathcal{L}_d)$ denotes the classical intersection number of $\mathcal{L}_1,\dots, \mathcal{L}_d$ on $Y$. Thus, the numerical equivalence class of $\mathcal{O}_Y(r,w)$ does not depend on the choice of $w\in|B|$, and we will just write $r$ for its numerical equivalence class.
\section{Admissible metrics on the curve}\label{sec_admissible}
In this section we recall Zhang's admissible pairing on curves introduced in \cite{Zha93}. Let $k$ be an algebraically closed field, $B$ any smooth projective curve over $k$, and $K=k(B)$ its function field. Moreover, let $X$ be a smooth projective geometrically connected curve of genus $g$ over $K$ having semistable reduction over $B$.

First, let us recall the notion of the metrized reduction graph. Let $\pi\colon \mathcal{X}\to B$ be the minimal regular model of $X$ over $B$. The polarized metrized reduction graph $\Gamma_v(X)$ of $X$ at $v\in |B|$ is defined as follows:
\begin{itemize}
	\item The dual graph $(V_v,E_v)$ of the fiber $\mathcal{X}_v=\pi^{-1}(v)$ is given as follows:
	\begin{itemize}
		\item Its vertex set $V_v$ consists of the irreducible components of $\mathcal{X}_v$.
		\item Its edge set $E_v$ consists of the nodes in $\mathcal{X}_v$, such that any edge $e$ connects the two vertices corresponding to the irreducible components meeting in the node corresponding to $e$.
	\end{itemize} 
	\item We denote $\ell\colon E_v\to \mathbb{R}_{\ge 0}$ for the constant function $\ell(e)=1$.
	\item We associate the metrized graph $\Gamma_v(X)=\Gamma_{(V_v,E_v)}$ to $(V_v,E_v)$ and the length function $\ell$ as described in Section \ref{sec_graphs}.
	\item The polarization of $\Gamma_v(X)$ is given by
	\[K(v)=\sum_{p\in V_v}(n_p+2g_p-2)p,\]
	where $n_p$ denotes the valency of the vertex $p$ and $g_p$ denotes the genus of the normalization of the irreducible component corresponding to $p\in V_v$.
\end{itemize}
The polarized metrized graph $\Gamma_v(X)=(\Gamma_v(X),K(v))$ has genus $g$. We write $g_v$ for the canonical Green function associated to $\Gamma_v(X)$. There is a canonical map $R_v\colon X(K)\to \Gamma_v(X)$ given as follows: We can uniquely extend any point $P\in X(K)$ to a section $\widetilde{P}\colon B\to\mathcal{X}$ of $\pi$. We set $R_v(P)$ to be the point of $\Gamma_v(X)$ that corresponds to the vertex of the dual graph, which in turn corresponds to the irreducible component of $\mathcal{X}_v$ where $\widetilde{P}$ intersects $\mathcal{X}_v$.

As outlined by Yuan in \cite[Appendix]{Yua21}, Zhang's construction of admissible pairings \cite{Zha93} induces canonical admissible adelic metrics on the canonical bundle $\omega$ on $X$, on the bundle $\mathcal{O}_X(D)$ associated to any divisor $D$ on $X$ and on the diagonal bundle $\mathcal{O}_{X^2}(\Delta)$ on $X^2$. For our applications it is enough to describe the canonical metric on $\mathcal{O}_{X^2}(\Delta)$ for rational points $P,Q\in X(K)$. This metric satisfies
\begin{align}\label{equ_diagonalmetric}
	\log\|1_\Delta(P,Q)\|_v=-i_v(P,Q)-g_{v}(R_v(P),R_v(Q)),
\end{align}
where $v\in |B|$ is any place, $P,Q\in X(K)$ are arbitrary distinct $K$-rational points, $1_\Delta$ denotes the canonical section on $\mathcal{O}_{X^2}(\Delta)$, and $i_v(P,Q)$ is the usual intersection index on $\mathcal{X}_v$ of the unique sections $\widetilde{P}$ and $\widetilde{Q}$ of $\pi$ extending $P$ and $Q$. For the general definition of the canonical metric on $\mathcal{O}_{X^2}(\Delta)$ we refer to \cite[Theorem A.1]{Yua21}.

The canonical admissible metric on $\omega$ is chosen such that the canonical isomorphism $s^*\mathcal{O}_{X^2}(\Delta)\cong \omega^{\otimes-1}$ is an isometry, where $s\colon X\to X^2,~x\mapsto(x,x)$ denotes the diagonal embedding. Similarly, the admissible metric on $\mathcal{O}_X(P)$ for a point $P$ is given such that the canonical isomorphism $s_P^*\mathcal{O}_{X^2}(\Delta)\cong \mathcal{O}_X(P)$ is an isometry, where $s_P\colon X\to X^2,~x\mapsto (x,P)$ is the embedding associated to $P$. We let $\omega_a$, $D_a$ and $\Delta_a$ denote the canonical admissible adelic line bundles $\omega$, $\mathcal{O}_X(D)$ and $\mathcal{O}_{X^2}(\Delta)$. According to \cite[Proposition 3.5.1]{Zha10} these are integrable line bundles.

In general we call an adelic line bundle $\mathcal{L}_a=\overline{\mathcal{L}}$ on $X$ \emph{admissible} if its metric differs only by a constant from the metric on $D_a$ at every place $v\in |B|$, where $D$ is a divisor on $X$ satisfying $\mathcal{O}_X(D)\cong \mathcal{L}$.
The \emph{admissible pairing} $(\mathcal{L}_a,\mathcal{M}_a)$ of two admissible line bundles $\mathcal{L}_a$ and $\mathcal{M}_a$ on $X$ is defined to be the intersection number $\mathcal{L}_a\cdot \mathcal{M}_a$ of the integrable line bundles. For two different points $P,Q\in X(K)$ we deduce from Equation (\ref{equ_diagonalmetric}), that
\begin{align}\label{equ_admissiblepairing}
	(P_a,Q_a)=\sum_{v\in |B|}(i_v(P,Q)+ g_v(R_v(P),R_v(Q))).
\end{align}
Moreover, Zhang obtained in \cite[Theorem 4.4]{Zha93} for the self-intersection number of $\omega_a$
\begin{align}\label{equ_omegaepsilon}
	\omega_a^2=(\omega_a,\omega_a)=\omega_{\mathcal{X}/B}^2-\sum_{v\in |B|}\epsilon(\Gamma_v(X)),
\end{align}
where $\omega_{\mathcal{X}/B}^2$ denotes the usual self-intersection number of the relative dualizing sheaf $\omega_{\mathcal{X}/B}$.

We also recall the adjunction formula \cite[Theorem 4.2]{Zha93}
\begin{align}\label{equ_adjunction}
	(\omega_a,P_a)=-(P_a,P_a)
\end{align}
for any point $P\in X(K)$. This can be checked by applying formula (\ref{equ_intersection}) to the intersection product $-\Delta_a\cdot\Delta_a\cdot p_1^* P_a$, where $p_1\colon X^2\to X$ denotes the projection to the first factor. Indeed, applying (\ref{equ_intersection}) with respect to the canonical section $1_\Delta\in H^0(X^2,\mathcal{O}_{X^2}(\Delta))$ yields the left hand side and applying (\ref{equ_intersection}) with respect to the canonical section $p_1^* 1_P\in H^0(X^2,p_1^*\mathcal{O}_{X}(P))$ yields the right hand side. In both cases, the integral in (\ref{equ_intersection}) vanishes, as worked out in \cite[Section 3.5]{Zha10}, see also Section \ref{sec_hodgeindex} for similar computations.

If $\mathcal{L}_a$ is an admissible line bundle on $X$ with $\deg \mathcal{L}=0$, then $\mathcal{L}$ defines a point in the Jacobian variety $\operatorname{Pic}^0(X)$ of $X$ and by \cite[(5.4)]{Zha93} its N\'eron--Tate height, defined as in the introduction, is equal to
\begin{align}\label{equ_heightnt}
	h_{\operatorname{NT}}(\mathcal{L})=-(\mathcal{L}_a,\mathcal{L}_a).
\end{align}

\section{Admissible metrics on the Jacobian variety}\label{sec_jacobian}
Continuing the notation from the last section, we discuss admissible metrics on the Jacobian variety $J=\operatorname{Pic}^0(X)$ of $X$. We will compare them to the admissible metrics on the curve. Here, we additionally assume that $X(K)$ is non-empty. We refer to \cite[Section 2]{Zha95} and \cite{Yam17} for more details and to \cite[Section 1]{Mil86} for the definition of the Jacobian variety.

We fix a point $P\in X(K)$ and we denote $\Theta$ for the divisor given by the image of the map
\[X^{g-1}\to J,\quad (P_1,\dots,P_{g-1})\mapsto [P_1+\dots+P_{g-1}-(g-1)P],\]
which is ample by \cite[Theorem 6.6]{Mil86}. Hence, we get an ample and symmetric line bundle $\mathcal{L}=\mathcal{O}_{J}(\Theta)\otimes[-1]^*\mathcal{O}_{J}(\Theta)$ on $J$, where in general $[n]\colon J\to J$ denotes the multiplication-by-$n$ morphism on $J$ for any $n\in\mathbb{Z}$. Thus, we can choose an isomorphism $\phi\colon \mathcal{L}^{\otimes 4}\cong [2]^*\mathcal{L}$. Let $(\widetilde{J},\widetilde{\mathcal{L}})$ be a model of $(J,\mathcal{L})$ over $B$. As $\mathcal{L}$ is ample, we may assume that $\widetilde{\mathcal{L}}$ is nef. There exists an open non-empty subset $U\subseteq B$, such that the maps $[2]$ and $\phi$ extend to maps $[2]\colon \widetilde{J}_U\to \widetilde{J}_U$ and $\phi_U\colon \widetilde{\mathcal{L}}_{U}^{\otimes 4}\to [2]^*\widetilde{\mathcal{L}}_{U}$ on the base changes $\widetilde{J}_U$ and $\widetilde{\mathcal{L}}_U$ of $\widetilde{J}$ and $\widetilde{\mathcal{L}}$ along the embedding $U\to B$. We set $(\widetilde{J}_0,\widetilde{\mathcal{L}}_0)=(\widetilde{J},\widetilde{\mathcal{L}})$ and inductively, we construct models $(\widetilde{J}_n,\widetilde{\mathcal{L}}_n)$ of $(J,\mathcal{L}^{\otimes 4^n})$, such that $[2]\colon J\to J$ extends to a map $f_n\colon \widetilde{J}_{n+1}\to \widetilde{J}_n$ and $\widetilde{\mathcal{L}}_{n+1}=f_n^*\widetilde{\mathcal{L}}_n$. Then the sequence of metrics on $\mathcal{L}$ associated to the models $(\widetilde{J}_n,\widetilde{\mathcal{L}}_n)$ converges to an adelic metric $\|\cdot\|$ on $\mathcal{L}$, as the restriction of $\|\cdot\|$ to $U$ is given by the model metric of $(\widetilde{J}_U,\widetilde{\mathcal{L}}_{U})$ and as Zhang proved in \cite[Theorem (2.2)]{Zha95}, $\|\cdot\|_v$ is continuous and bounded for all $v\in |B|$. The adelic line bundle  $\hat{\mathcal{L}}=(\mathcal{L},\|\cdot\|)$ is called an \emph{admissible adelic line bundle}. Since $\widetilde{\mathcal{L}}_n$ is nef for all $n\ge 0$, we deduce that $\hat{\mathcal{L}}$ is a nef adelic line bundle. We remark that by construction
\begin{align}\label{equ_NT}
	h_{\operatorname{NT}}(x)=h_{\hat{\mathcal{L}}}(x):=\frac{\widehat{\deg}(\hat{\mathcal{L}}|_{x})}{[K(x):K]}
\end{align}
for all $x\in J(\overline{K})$, where we also denote by $x$ the image of $x\colon \operatorname{Spec}(\overline{K})\to J$.

The following lemma gives a comparison of the numerical equivalence classes of the admissible metrized line bundles obtained on the curve and on the Jacobian variety. We recall from Section 3 that if $r$ is a real number, we also write $r$ for the numerical equivalence class of the trivial line bundle equipped with the canonical metric multiplied by $e^{-r}$ at one specified place and with the canonical metrics at all of the other places.
\begin{Lem}\label{lem_nefbundle}
	Consider the map
	\[f\colon X^2\to J,\quad (P_1,P_2)\mapsto [(g-1)(P_1+P_2)-\omega]\]
	and write $p_i\colon X^2\to X$ for the projection to the $i$-th factor.
	The nef adelic line bundle $f^*\hat{\mathcal{L}}$ on $X^2$ lies in the numerical equivalence class of
	\[(g-1)(g+1)(p_1^*\omega_a+p_2^*\omega_a)-2(g-1)^2\Delta_a-\omega_a^2.\]
\end{Lem}
\begin{proof}
	A theorem by Carney \cite[Theorem 3.2]{Car20} implies that an integrable line bundle $\overline{\mathcal{M}}$ on a smooth projective variety $Y$ over $K$ is numerically trivial if and only if $\widehat{\deg}(\overline{\mathcal{M}}|_{y})=0$ for all $y\in Y(\overline{K})$.
	Thus, it is enough to show that
	\[\widehat{\deg}\left(f^*\hat{\mathcal{L}}|_{x}\right)=\widehat{\deg}\left(((g-1)(g+1)(p_1^*\omega_a+p_2^*\omega_a)-2(g-1)^2\Delta_a-\omega_a^2)|_{x}\right)\]
	for all $x\in X^2(\overline{K})$. We fix an $x\in X^2(\overline{K})$. By the projection formula (\ref{equ_projeciton-field-extension}) it is enough to prove the above equality after replacing $K$ by a finite field extension $K'$ and $X$ by its base change $X_{K'}$. Thus, we may assume $x\in X^2(K)$.
	If we write $x=(P_1,P_2)$ with $P_1,P_2\in X(K)$, we get by Equations (\ref{equ_projection1}) and (\ref{equ_NT}) that
	\[\widehat{\deg}\left(f^*\hat{\mathcal{L}}|_{x}\right)=\widehat{\deg}\left(\hat{\mathcal{L}}|_{f(P_1,P_2)}\right)=h_{\operatorname{NT}}((g-1)(P_1+P_2)-\omega).\]
	Further, we may compute using Equations (\ref{equ_adjunction}) and (\ref{equ_heightnt}) that
	\begin{align*}
		&h_{\operatorname{NT}}((g-1)(P_1+P_2)-\omega)\\
		&=-((g-1)(P_{1,a}+P_{2,a})-\omega_a,(g-1)(P_{1,a}+P_{2,a})-\omega_a)\\
		&=(g-1)(g+1)(\omega_a,P_{1,a}+P_{2,a})-2(g-1)^2(P_{1,a},P_{2,a})-\omega_a^2\\
		&=\widehat{\deg}\left(((g-1)(g+1)(p_1^*\omega_a+p_2^*\omega_a)-2(g-1)^2\Delta_a-\omega_a^2)|_{(P_1,P_2)}\right).
	\end{align*}
	This completes the proof of the lemma.
\end{proof}
\section{An application of the arithmetic Hodge index theorem}\label{sec_hodgeindex}
The goal of this section is to deduce a lower bound for $\omega_a^2$ as an application of the arithmetic Hodge index theorem over function fields by Carney \cite[Theorem~3.1]{Car20} for $X^2$. This is motivated by the analogous result in the number field case found by the third author in \cite[Theorem 1.2]{Wil19}. We retain the notation from the previous sections.
We first recall the inequality part of the arithmetic Hodge index theorem over function fields for the variety $X^2$. If $\overline{\mathcal{N}}$ and $\overline{\mathcal{M}}$ are integrable line bundles on $X^2$ such that
\begin{enumerate}[(i)]
	\item\label{condition_i} $\overline{\mathcal{N}}$ is nef,
	\item\label{condition_ii} $\mathcal{N}$ is big and
	\item\label{condition_iii} their usual intersection number satisfies $\mathcal{M}\cdot\mathcal{N}=0$,
\end{enumerate}
then $\overline{\mathcal{M}}\cdot\overline{\mathcal{M}}\cdot \overline{\mathcal{N}}\le 0$. As the statement only depends on the numerical equivalence class of $\overline{\mathcal{N}}$, it is enough to check \ref{condition_i} for an integrable line bundle numerically equivalent to $\overline{\mathcal{N}}$. It is even enough to check it for an integrable line bundle numerically equivalent to any positive multiple of $\overline{\mathcal{N}}$.
\begin{Pro}\label{pro_lowerbound}
	Any smooth projective geometrically connected curve $X$ of genus $g\ge 2$ defined over $K$ satisfies
	\[\omega_a^2\ge \frac{\max(2,g-1)}{2g+1}\sum_{v\in |B|}\varphi(\Gamma_v(X)).\]
	If $\operatorname{char}~k=0$ or $X$ is hyperelliptic, we may replace $\max(2,g-1)$ by $2(g-1)$.
\end{Pro}
\begin{proof}
	If $\operatorname{char}~k=0$, this follows from \cite[Section 1.4]{Zha10} and if $X$ is hyperelliptic, equality in fact holds, by \cite[Corollary 1.3.3]{Zha10}. As all curves of genus $g=2$ are hyperelliptic, it remains to prove $\omega_a^2\ge \tfrac{g-1}{2g+1}\sum_{v\in B}\varphi(\Gamma_v(X))$ for $g\ge 3$. To simplify the notation, we write $\omega_{12}=p_1^*\omega+ p_2^*\omega$, where $p_i\colon X^2\to X$ denotes the projection to the $i$-th factor. Further, we set $\hat{\omega}_{12}=p_1^*\omega_a+p_2^*\omega_a$. We would like to apply the arithmetic Hodge index theorem to the integrable line bundles
	\[\overline{\mathcal{N}}=(g+1)\hat{\omega}_{12}-2(g-1)\Delta_a-\frac{\omega_a^2}{g-1},\quad \overline{\mathcal{M}}=(g-1)\Delta_a-\hat{\omega}_{12}\]
	on $X^2$. Strictly speaking, the constant $\omega_a^2$ only defines an integrable line bundle up to numerical equivalence in our notation, and hence $\overline{\mathcal{N}}$ is similarly only well-defined up to numerical equivalence. But this is sufficient for the application to the arithmetic Hodge index theorem.
	
	By Lemma \ref{lem_nefbundle} the numerical equivalence class $(g-1)\overline{\mathcal{N}}$ can be represented by a nef adelic line bundle. Thus, condition \ref{condition_i} is satisfied.
	This also implies, that the underlying line bundle $\mathcal{N}$ is nef on $X^2$, such that we have $\operatorname{vol}(\mathcal{N})=\mathcal{N}^2$. We compute the volume explicitly:
	\begin{align*}
		\operatorname{vol}(\mathcal{N})&=\left((g+1)(\omega_{12})-2(g-1)\mathcal{O}_{X^2}(\Delta)\right)^2\\
		&=2(g+1)^2p_1^*\omega\cdot p_2^*\omega-4(g^2-1)\omega_{12}\cdot\mathcal{O}_{X^2}(\Delta)+4(g-1)^2\mathcal{O}_{X^2}(\Delta)^2\\
		&=\left(4(g+1)^2(g-1)-8(g^2-1)-4(g-1)^2\right)\deg \omega=8g(g-1)^3>0.
	\end{align*}
	Hence, condition \ref{condition_ii} is also satisfied. To check condition \ref{condition_iii}, we calculate that
	\begin{align*}
		\mathcal{M}\cdot\mathcal{N}&=\left((g+1)(\omega_{12})-2(g-1)\mathcal{O}_{X^2}(\Delta)\right)\cdot \left((g-1)\mathcal{O}_{X^2}(\Delta)-\omega_{12}\right)\\
		&=-2(g+1)p_1^*\omega\cdot p_2^*\omega+(g+3)(g-1)\omega_{12}\mathcal{O}_{X^2}(\Delta)-2(g-1)^2\mathcal{O}_{X^2}(\Delta)^2\\
		&=(-4(g+1)(g-1)+2(g+3)(g-1)+2(g-1)^2)\deg\omega=0.
	\end{align*}
	Thus, the arithmetic Hodge index theorem implies that $\overline{\mathcal{M}}\cdot\overline{\mathcal{M}}\cdot\overline{\mathcal{N}}\le 0$. We now compute the left hand side to be
	\begin{align}\label{equ_MMN}
		\overline{\mathcal{M}}\cdot\overline{\mathcal{M}}\cdot\overline{\mathcal{N}}=&-2(g-1)^3\Delta_a^3+(g+5)(g-1)^2\hat{\omega}_{12}\cdot\Delta_a^2-2(g+2)(g-1)\hat{\omega}_{12}^2\cdot \Delta_a\\
		&+(g+1)\hat{\omega}_{12}^3-\frac{\mathcal{M}\cdot\mathcal{M}}{g-1}\omega_a^2.\nonumber
	\end{align}
	By a similar computation as above we have
	\[\mathcal{M}\cdot\mathcal{M}=2p_1^*\omega\cdot p_2^*\omega-2(g-1)\omega_{12}\cdot\mathcal{O}_{X^2}(\Delta)+(g-1)^2\mathcal{O}_{X^2}(\Delta)^2=-2(g-1)^3.\]
	By symmetry and by formula (\ref{equ_projection2}) we obtain
	\[\hat{\omega}_{12}^3=6(p_1^*\omega_a\cdot p_1^*\omega_a\cdot p_2^*\omega_a)=6\cdot \omega_a^2\cdot \deg \omega=12(g-1)\omega_a^2.\]
	In the same way, we obtain $p_1^*\omega_a\cdot p_1^*\omega_a\cdot\Delta_a=\omega_a^2$. Let $s\colon X\to X^2,~x\mapsto (x,x)$ the diagonal embedding. Then we can compute by the recursion formula for the intersection number (\ref{equ_intersection}) and by the definition of the metric on $\Delta_a$ in (\ref{equ_diagonalmetric}) that
	\[p_1^*\omega_a \cdot p_2^*\omega_a\cdot \Delta_a=s^*p_1^*\omega_a\cdot s^*p_2^*\omega_a+\sum_{v\in|B|}\int_{X^2(\overline{K}_v)}g_v c_1(p_1^*\omega_a)c_1(p_2^*\omega_a)=\omega_a^2,\]
	where the integral vanishes by the definition of the Green function $g_v$. See also \cite[Section 3.5]{Zha10} where it is also explained in more detail why we can replace $-\log\|1_\Delta\|_v$ by $g_v$. Hence, we can conclude again using symmetry	
	\[\hat{\omega}_{12}^2\cdot \Delta_a=2(p_1^*\omega_a\cdot p_1^*\omega_a\cdot \Delta_a+p_1^*\omega_a\cdot p_2^*\omega_a\cdot \Delta_a)=4\omega_a^2.\]
	In a similar way, we obtain from (\ref{equ_intersection}) and (\ref{equ_diagonalmetric}) that
	\[p_1^*\omega_a \cdot \Delta_a\cdot \Delta_a=s^*p_1^*\omega_a\cdot s^*\Delta_a+\sum_{v\in|B|}\int_{X^2(\overline{K}_v)}g_v c_1(p_1^*\omega_a)c_1(\Delta_a)=-\omega_a^2,\]
	where the integral vanishes as in \cite[Lemma 3.5.2]{Zha10}. We conclude by symmetry that
	\[\hat{\omega}_{12}\cdot\Delta_a^2=2(p_1^*\omega_a\cdot\Delta_a\cdot\Delta_a)=-2\omega_a^2.\]
	Finally, we compute the self-intersection number $\Delta_a^3$ to be
	\[\Delta_a^3=(s^*\Delta_a)^2+\sum_{v\in|B|}\int_{X^2(\overline{K}_v)} g_v c_1(\Delta_a)^2=\omega_a^2-\sum_{v\in|B|}\varphi(\Gamma_v(X)),\]
	where the last equality follows from \cite[Lemma 3.5.4]{Zha10}. Substituting everything into Equation (\ref{equ_MMN}) yields
	\[0\ge \overline{\mathcal{M}}\cdot\overline{\mathcal{M}}\cdot\overline{\mathcal{N}}=2(g-1)^3\sum_{v\in|B|}\varphi(\Gamma_v(X))-2(g-1)^2(2g+1)\omega_a^2,\]
	which is equivalent to the inequality we wanted.
\end{proof}
\section{Proof of Theorems \ref{mainthm} and \ref{thm_bogomolov}}\label{sec_proof}
In this section we prove Theorems \ref{mainthm} and \ref{thm_bogomolov}. We use the same notation as in the theorems.
By the semistable reduction theorem there exists a finite field extension $K'$ of $K$ such that $X_{K'}=X\otimes_K K'$ has semistable reduction over the normalization $B'$ of $B$ in $K'$. As $X_{K'}(\overline{K})= X(\overline{K})$, we may assume that $X$ has semistable reduction over $B$. Note that the base change to $X_{K'}$  does not change the statement of Theorem \ref{thm_bogomolov} as it modifies the values of $h_{\operatorname{NT}}$ and $\omega_a^2$ by the same factor $[K':K]$. (Note also that the computation of $\omega_a^2$ as described in Section \ref{sec_admissible} makes sense regardless of whether $X$ has semistable reduction over $K$.) We write $\pi\colon \mathcal{X}\to B$ for the minimal regular model of $X$ over $B$.

Fix an $\epsilon\ge 0$ and let $F=\{P_1,\dots, P_s\}\subseteq X(\overline{K})$ be any set of $s\ge 2$ geometric points, such that $h_{\operatorname{NT}}(j_D(P_i))\le \epsilon \omega_a^2$ for any $1\le i\le s$. If no such set exists, we have nothing to prove. Replacing $K$ by a finite field extension again, we can assume that $F\subseteq X(K)$. By the bilinearity of the pairing and by Equation (\ref{equ_heightnt}) one checks directly with the parallelogram law that
\begin{align*}
	&h_{\operatorname{NT}}(P_j+P_k-2D)+h_{\operatorname{NT}}(P_j-P_k)\\
	&=-(P_{j,a}+P_{k,a}-2D_a,P_{j,a}+P_{k,a}-2D_a)-(P_{j,a}-P_{k,a},P_{j,a}-P_{k,a})\\
	&=-2(P_{j,a}-D_a,P_{j,a}-D_a)-2(P_{k,a}-D_a,P_{k,a}-D_a)\\
	&=2h_{\operatorname{NT}}(j_D(P_j))+2h_{\operatorname{NT}}(j_D(P_k)),
\end{align*}
for all $j,k\le s$. Using the height bounds $h_{\operatorname{NT}}(j_D(P_i))\le \epsilon \omega_a^2$ for all $i\le s$ and also $h_{\operatorname{NT}}(P_j+P_k-2D)\ge 0$, we deduce that
\[2(P_{j,a},P_{k,a})-(P_{j,a},P_{j,a})-(P_{k,a},P_{k,a})\le 4\epsilon \omega_a^2.\]
Summing over all pairs of different points in $F$, we obtain
\begin{align}\label{equ_sum}
	-\sum_{j=1}^s(P_{j,a},P_{j,a})\le -\frac{1}{s-1}\sum_{j\neq k}^s(P_{j,a},P_{k,a})+2s\epsilon \omega_a^2.
\end{align}
We also write $F$ for the divisor $F=P_1+\dots+P_s$. Again by Equation (\ref{equ_heightnt}) we have
\[0\le h_{\operatorname{NT}}(s\omega-(2g-2)F)=-(s\omega_a-(2g-2)F_a,s\omega_a-(2g-2)F_a).\]
By bilinearity and symmetry of the pairing we may rewrite this inequality as
\[\omega_a^2\le \tfrac{4(g-1)}{s}\sum_{j=1}^s(\omega_a,P_{j,a})-\tfrac{4(g-1)^2}{s^2}\left(\sum_{j=1}^s(P_{j,a},P_{j,a})+\sum_{j\neq k}^s(P_{j,a},P_{k,a})\right).\]
As $(\omega_a,P_{j,a})=-(P_{j,a},P_{j,a})$ by the adjunction formula (\ref{equ_adjunction}), we obtain by an application of (\ref{equ_sum}) that
\begin{align}\label{equ_bound_s}
	\omega_a^2&\le\left(-\tfrac{4(g-1)}{s(s-1)}-\tfrac{4(g-1)^2}{s^2}-\tfrac{4(g-1)^2}{s^2(s-1)}\right)\sum_{j\neq k}^s (P_{j,a},P_{k,a})+\left(\tfrac{4(g-1)^2}{s^2}+\tfrac{4(g-1)}{s}\right)2s\epsilon\omega_a^2\\
	&=-\tfrac{4(g-1)g}{s(s-1)}\sum_{j\neq k}^s (P_{j,a},P_{k,a})+8(g-1)\tfrac{g-1+s}{s}\epsilon\omega_a^2\nonumber\\
	&\le -\tfrac{4(g-1)g}{s(s-1)}\sum_{j\neq k}^s (P_{j,a},P_{k,a})+4(g^2-1)\epsilon\omega_a^2\nonumber,
\end{align}
where the last inequality follows from the fact that $\frac{g-1+s}{s}$ is maximized at $s=2$ for $s\ge 2$.
Let us first assume $\epsilon\in \left[0,\frac{1}{4(g^2-1)}\right)$. Then we can solve for $\omega_a^2$ and get
\[\omega_a^2\le -\tfrac{4(g-1)g}{s(s-1)(1-4(g^2-1)\epsilon)}\sum_{j\neq k}^s (P_{j,a},P_{k,a}).\]
By Equation (\ref{equ_admissiblepairing}), we have 
\[(P_{j,a},P_{k,a})=\sum_{v\in |B|}(i_v(P_j,P_k)+g_v(R_v(P_j),R_v(P_k)))\ge \sum_{v\in |B|}g_v(R_v(P_j),R_v(P_k))\]
for $j\neq k$.
Thus, we conclude using Propositions \ref{pro_elkies} and \ref{pro_lowerbound} that
\begin{align}\label{equ_bound-two-sided}
	\tfrac{\max(2,g-1)}{2g+1}\sum_{v\in |B|}\varphi(\Gamma_v(X))\le\omega_a^2\le \tfrac{4(g-1)gc'(g)}{(s-1)(1-4(g^2-1)\epsilon)}\sum_{v\in |B|}\varphi(\Gamma_v(X)),
\end{align}
with $c'(g)$ as in Proposition \ref{pro_elkies}.

We first consider the case where $X$ has everywhere potentially good reduction.
As $X$ and $J$ have semistable reduction, they have good reduction at a place if they have potentially good reduction at this place.
Thus, we get $\omega_a^2=\omega_{\mathcal{X}/B}^2$ by Equation (\ref{equ_omegaepsilon}). The Noether formula states that $\omega_{\mathcal{X}/B}^2=12\deg \pi_* \omega_{\mathcal{X}/B}$ in this case, and it has been shown by Par\v{s}in \cite[Proposition 5]{Par68} when $\operatorname{char}~k=0$ and by Szpiro \cite[Theorem 1]{Szp81} when $\operatorname{char}~k>0$ that $\deg \pi_* \omega_{\mathcal{X}/B}>0$, as $X$ is non-isotrivial. On the other hand we have $\sum_{v\in |B|}\varphi(\Gamma_v(X))=0$ in this case, contradicting the above inequality. Hence there is no set $F$ as above if $X$ has everywhere potentially good reduction. Consequently, we obtain
\[\#\{P\in X(\overline{K})~|~h_{\operatorname{NT}}(P)\le \epsilon \omega_a^2\}\le 1\]
if $X$ has everywhere potentially good reduction.

Now we consider the case where $X$ does not have everywhere potentially good reduction and $\epsilon\in \left[0,\max\left(\frac{1}{8g},\frac{1}{9(g-1)}\right)\right)$. Note that Theorem \ref{thm_bogomolov} is trivially satisfied if there is no set $F=\{P_1,\dots,P_s\}\subset X(\overline{K})$ of cardinality $s\ge \max\{9(g-1),(g-1)^2\}$. Thus, we may assume $s\ge \max\{9(g-1),(g-1)^2\}$. This allows us to make a stronger estimate in the last step of (\ref{equ_bound_s}) and we get
\[\omega_a^2\le -\tfrac{4(g-1)g}{s(s-1)}\sum_{j\neq k}^s (P_{j,a},P_{k,a})+\min\{9(g-1),8g\}\epsilon\omega_a^2.\]
This leads to the following improved version of (\ref{equ_bound-two-sided}):
\[\tfrac{\max(2,g-1)}{2g+1}\sum_{v\in |B|}\varphi(\Gamma_v(X))\le\omega_a^2\le \tfrac{4(g-1)gc'(g)}{(s-1)(1-\min\{9(g-1),8g\}\epsilon)}\sum_{v\in |B|}\varphi(\Gamma_v(X)).\]
Since $\sum_{v\in |B|}\varphi(\Gamma_v(X))>0$ by (\ref{inequ_cinkir}), we can solve for $s$ and thereby conclude that
\begin{align}\label{equ_bounds}
	s\le \left\lfloor\frac{4(g-1)g (2g+1)c'(g)}{\max(2,g-1)(1-\min\{9(g-1),8g\}\epsilon)}\right\rfloor+1=:c(g,\epsilon).
\end{align}
We now compute this number $c(g,\epsilon)$ explicitly. For $g\le 3$, one obtains
\[c(2,\epsilon)=\left\lfloor \frac{75}{1-9\epsilon}\right\rfloor+1,\qquad c(3,\epsilon)=\left\lfloor \frac{3920}{17(1-18\epsilon)}\right\rfloor +1.\]
In general, $c(g,\epsilon)=\left\lfloor\frac{16g^4+36g^2-26g-2}{(g-1)^2(1-\min\{9(g-1),8g\}\epsilon)}\right\rfloor+1$ for $g\ge 4$.

If $J$ has everywhere potentially good reduction, then every reduction graph $\Gamma_v(X)$ is a tree. Hence by Proposition \ref{pro_elkies}, we may replace $c'(g)$ in (\ref{equ_bounds}) by $c'^{tr}(g)$, such that we obtain $s\le c^{tr}(g,\epsilon)$ with 
\[c^{tr}(2,\epsilon)=\left\lfloor\frac{10}{1-9\epsilon}\right\rfloor+1\quad \text{and} \quad c^{tr}(g,\epsilon)=\left\lfloor\frac{4g^3-4g^2+g+2}{(g-1)(1-\min\{9(g-1),8g\}\epsilon)}\right\rfloor+1\]
for $g\ge 3$. If $\operatorname{char}~k=0$ or $X$ is hyperelliptic, we may replace $\max(2,g-1)$ in (\ref{equ_bounds}) by $2g-2$ by Proposition \ref{pro_lowerbound}. Thus, for $g\ge 3$ we get $s\le \left\lfloor \frac{c(g,\epsilon)+1}{2}\right\rfloor$, respectively $s\le \left\lfloor \frac{c^{tr}(g,\epsilon)+1}{2}\right\rfloor$ if $J$ has everywhere potentially good reduction. 
This completes the proof of Theorem \ref{thm_bogomolov}.
As geometric torsion points $x\in J(\overline{K})_{\operatorname{tors}}$ have N\'eron--Tate height $h_{\operatorname{NT}}(x)=0$, we recover Theorem \ref{mainthm} by setting $\epsilon=0$ and $c(g)=c(g,0)$ and $c^{tr}(g)=c^{tr}(g,0)$.

\section{Proof of Corollary 1.2}\label{sec_cor}
In this section we prove Corollary \ref{cor}. Let $K$ be an algebraically closed field with prime field $\mathbb{F}$ and $X/K$ a smooth projective connected curve of genus $g\ge 2$, which cannot be defined over an algebraic extension of $\mathbb{F}$. We can choose a subfield $K_0\subseteq K$ which is finitely generated over $\mathbb{F}$ such that $X=X_0\otimes_{K_0}K$ for $X_0/K_0$, and we choose $K_0$ such that $\operatorname{trdeg}(K_0/\mathbb{F})$ is minimal among these fields. By assumption $\operatorname{trdeg}(K_0/\mathbb{F})\ge 1$. We choose a subfield $k_0\subseteq K_0$ with $\operatorname{trdeg}(K_0/k_0)=1$ and denote $k$ for the algebraic closure of $k_0$ in $K$. We define $k(B)$ to be the compositum of $k$ and $K_0$ in $K$. Since $k(B)$ is a finitely generated extension of $k$ with $\operatorname{trdeg}(k(B)/k)=1$, there exists a smooth projective connected curve $B$ defined over $k$ such that $k(B)$ is the function field of $B$, see for example \cite[Theorem I.6.9]{Har77}. We write $\overline{k(B)}$ for the algebraic closure of $k(B)$ in $K$.

Now let $D\in\operatorname{Div}^1(X)$ be a degree $1$ divisor on $X$ and $P_1,\dots,P_s\in X(K)$ any pairwise different points such that
\[j_D(P_1),\dots,j_D(P_s)\in j_D(X(K))\cap J(K)_{\operatorname{tors}}.\]
We have to show that $s\le c(g)$, where $c(g)$ is as in Theorem \ref{mainthm}. We choose a subfield $K_1\subseteq K$ which is finitely generated over $\overline{k(B)}$ such that $P_1,\dots, P_s$ and $D$ are defined over $K_1$. We write $X_1=X_0\otimes_{K_0}K_1$ and $J_1=\operatorname{Pic}^0(X_1)$ for its Jacobian. We denote by $P_1,\dots, P_s$ and $D$ also the points and the divisor on $X_1$ such that their base change to $X$ induce $P_1,\dots, P_s$ and $D$.
Let $V$ be a quasi-projective regular irreducible variety defined over $\overline{k(B)}$ with function field $K_1$. Let $\mathcal{X}_1$ be a model of $X_1$ over $V$. This can be obtained by choosing an embedding $X_1\subseteq \mathbb{P}^N_{K_1}$ and taking the Zariski closure of $X_1$ in $\mathbb{P}^N_{\overline{k(B)}}\times V$.
After replacing $V$ by a Zariski open dense subset, we may assume that
\begin{enumerate}[(i)]
	\item For every point $v\in V$ the fiber $\mathcal{X}_{1,v}$ is a smooth projective geometrically connected curve.
	\item No fiber $\mathcal{X}_{1,v}$ is definable over $k$.
	\item The specializations $P_{1,v},\dots P_{s,v}\in \mathcal{X}_{1,v}(\overline{k(B)})$ of the points $P_1,\dots,P_s$ are pairwise different for all points $v\in V$.
\end{enumerate}
Let $\mathcal{J}_1\to V$ be the Jacobian of the family $\mathcal{X}_1\to V$. The specialization $D_v$ of the divisor $D$ on $X_1$ defines a degree $1$ divisor on $\mathcal{X}_{1,v}$ for every $v\in V$ and hence an Abel--Jacobi embedding $j_{D_v}\colon \mathcal{X}_{1,v}\to \mathcal{J}_{1,v},~P\to P-D_v$. By construction taking the specialization commutes with the Abel--Jacobi maps, that is $j_{D_v}(P_v)=j_D(P)_v$, where $P_v$ is the specialization of a point $P\in X_1(K_1)$ at $v\in V$ and $j_D(P)_v$ is the specialization of the point $j_D(P)\in J_1(K_1)$. Since specialization respects the group structures of $J_1$ and $\mathcal{J}_{1,v}$, it follows that $j_{D_v}(P_{1,v}),\dots,j_{D_v}(P_{s,v})\in \mathcal{J}_{1,v}(\overline{k(B)})_{\operatorname{tors}}$.

We fix a $v\in V(\overline{k(B)})$. Then $\mathcal{X}_{1,v}$ is defined over a finite extension $K'$ of $k(B)$ in $\overline{k(B)}$. If $B'$ denotes the normalization of $B$ in $K'$, we get that $K'=k(B')$ is the function field of a smooth projective connected curve $B'$ over $k$. We write $X'$ for the curve such that $\mathcal{X}_{1,v}=X'\otimes_{k(B')}\overline{k(B)}$ and $J'=\operatorname{Pic}^0(X')$ for its Jacobian. Note that by (ii) the curve $X'$ is non-isotrivial. By Theorem \ref{mainthm} we have
\[\#j_D(\mathcal{X}_{1,v}(\overline{k(B)}))\cap \mathcal{J}_v(\overline{k(B)})_{\operatorname{tors}}=\#j_D(X'(\overline{k(B)}))\cap J'(\overline{k(B)})_{\operatorname{tors}}\le c(g).\]
Since $j_{D_v}(P_{1,v}),\dots,j_{D_v}(P_{s,v})\in j_D(\mathcal{X}_{1,v}(\overline{k(B)}))\cap \mathcal{J}_v(\overline{k(B)})_{\operatorname{tors}}$ are pairwise distinct points, we conclude that $s\le c(g)$. This completes the proof of Corollary \ref{cor}.

\section*{Acknowledgment}
	The authors would like to thank Matt Baker, Alex Carney,
	Robin de Jong, Xander Faber, Myrto Mavraki, and Lucia Mocz for their helpful conversations and useful insights on the admissible pairing. We also thank Jiawei Yu for pointing out an improvement in the bound for $\epsilon$ in Theorem \ref{thm_bogomolov}, as well as the editor for the suggestion of adding Corollary \ref{cor}. We thank the anonymous referees for their very careful reading and useful comments on two drafts of this paper.

\end{document}